\documentclass{amsart}

\usepackage{amssymb}

\usepackage[pdftitle={S. Lord,  F. A. Sukochev -- Noncommutative Residues and a Characterisation of the Noncommutative Integral},
            pdfauthor={S. Lord, F. A. Sukochev},
            pdfsubject={2000MSC Primary: 46L51, 47B10, 58J42 Secondary: 46L87},
            pdfkeywords={Dixmier Trace, Zeta Functions, Noncommutative Integral ,
             						Noncommutative Geometry, Normal, Noncommutative Residue},
            pdfdisplaydoctitle=true,
            pdfpagelayout=SinglePage]{hyperref}

\numberwithin{equation}{section} 

\theoremstyle{plain}
\newtheorem{thm}{Theorem}[section]
\newtheorem*{thm*}{Theorem}
\newtheorem{prop}[thm]{Proposition}
\newtheorem*{prop*}{Proposition}
\newtheorem{lemma}[thm]{Lemma}
\newtheorem*{lemma*}{Lemma}
\newtheorem{cor}[thm]{Corollary}
\newtheorem*{cor*}{Corollary}
\theoremstyle{definition}
\newtheorem{dfn}[thm]{Definition}
\newtheorem*{dfn*}{Definition}
\newtheorem{rems}[thm]{Remark}
\newtheorem*{rems*}{Remark}
\newtheorem{ex}[thm]{Example}
\newtheorem*{ex*}{Example}
\newtheorem*{conj*}{Conjecture}
\newtheorem*{cond*}{Condition}
\newtheorem*{not*}{Notation}

\newcommand{\bl}{\ensuremath{\xi}}

\DeclareMathOperator{\Dom}{Dom}   
\DeclareMathOperator{\Tr}{Tr}     
\DeclareMathOperator{\tr}{tr}     

\newlength{\indentation} 

\newcounter{propcount}

\newcounter{prop2count}
\newenvironment{prop2list}[3]{\begin{list}{\normalfont{(\roman{prop2count})}}
	{\usecounter{prop2count}
	\setlength{\topsep}{4pt} \setlength{\itemsep}{#3pt}
	\setlength{\parsep}{2pt} \setlength{\labelsep}{3mm}
	\setlength{\leftmargin}{#1mm}
	\setlength{\rightmargin}{#2mm}}}
       {\end{list}}

\newlength{\tablength} 

\newcommand{\nm}[1]{\mbox{\ensuremath{\| #1 \|}}}
\newcommand{\fa}{\ensuremath{\ \, \forall \,}}

\newcommand{\inset}[2]{\ensuremath{\{ #1 \, | \, #2 \} }}
\newcommand{\inprod}[2]{\ensuremath{\langle #1 , #2 \rangle}}

\newcommand{\ceil}[1]{\ensuremath{\left\lceil #1 \right\rceil}}
\newcommand{\floor}[1]{\ensuremath{\left\lfloor #1 \right\rfloor}}

\newcommand{\CC}{\ensuremath{\mathbb{C}}}

\newcommand{\TT}{\ensuremath{\mathbb{T}}}
\newcommand{\SB}{\ensuremath{\mathbb{S}}}
\newcommand{\ZZ}{\ensuremath{\mathbb{Z}}}

\newcommand{\NN}{\ensuremath{\mathbb{N}}}

\begin{document}

\title[Residues and a Characterisation of the Noncomm. Integral]{Noncommutative Residues and a Characterisation of the Noncommutative Integral}


\author{Steven Lord}
\address{School of Mathematical Sciences \\
University of Adelaide \\
Adelaide 5005 \\
Australia. (Corresponding Author).}
\thanks{This research was supported by the Australian Research Council}
\curraddr{}
\email{steven.lord@adelaide.edu.au}

\author{Fedor A.~Sukochev}
\address{School of Mathematics and Statistics \\
University of New South Wales \\
Sydney 2052 \\
Australia.}
\curraddr{}
\email{f.sukochev@unsw.edu.au}


\date{}

\dedicatory{}

\commby{M.~Varghese}

\begin{abstract}
We continue the study of the relationship between Dixmier traces and noncommutative residues initiated by A.~Connes.
The utility of the residue approach to Dixmier traces
is shown by a characterisation
of the noncommutative integral in Connes' noncommutative geometry (for a wide class of Dixmier traces) as a generalised limit of
vector states associated to the eigenvectors of a compact operator (or an unbounded operator with compact resolvent), i.e.~as a generalised quantum limit.
Using the characterisation, a criteria
involving the eigenvectors of a compact operator and the projections of a von Neumann subalgebra of bounded operators is given so that the noncommutative integral associated to the compact operator is \emph{normal}, i.e.~satisfies a monotone convergence theorem, for the von Neumann subalgebra.
\end{abstract}


\maketitle


\section{Introduction}

For a separable complex Hilbert space $H$ denote by $\mu_{n}(T)$, $n \in \NN$, the singular values of a positive compact operator $T$, (\cite{S}, \S 1).  A.~Connes introduced the association between a generalised zeta function,
$$
\zeta_T(s) := \Tr(T^s) = \sum_{n=1}^\infty \mu_{n}(T)^s ,
$$
and the logarithmic divergence of the partial sums,
$$
\left\{ \sum_{n=1}^{N} \mu_{n}(T) \right\}_{N=1}^\infty ,
$$
with the result that
\begin{equation} \label{eq:res2}
\lim_{s \to 1^+} (s-1) \zeta_T(s) = \lim_{N \to \infty} \frac{1}{\log(1+N)} \sum_{n=1}^{N} \mu_{n}(T)
\end{equation}
if either limit exists, (\cite{CN}, p.~306).  In \cite{LSS}, with co-author A.~Sedaev, we showed the right
hand side of equation (\ref{eq:res2}) is the Dixmier trace for Connes' notion of measurable operator, i.e.~an operator $T \in \mathcal{L}^{1,\infty}:=\inset{T}{\nm{T}_{1,\infty}:=\sup_{n \in \NN} \log(1+n)^{-1} \sum_{j=1}^n \mu_j(T) < \infty}$ is called measurable if the value of a Dixmier trace $\Tr_\omega(T)$, \cite{Dix}, (\cite{C3}, p.~674), is independent
of the `invariant mean' (dilation invariant state) $\omega$ on $\ell^\infty$, (\cite{CN}, Def 7 p.~308). 
As a result
\begin{equation} \label{eq:res0}
\Tr_\omega(T) = \lim_{s \to 1^+} (s-1) \zeta_T(s)
\end{equation}
enables the calculation of the Dixmier trace of any measurable operator $0 < T \in \mathcal{L}^{1,\infty}$ as the residue at $s=1$ of the zeta function $\zeta_T$.  We should note that it was Connes, in
(\cite{CN}, pp.~303-308), that showed a dilation invariant state on
$\ell^\infty$ was sufficient to define a Dixmier trace.  In practice
Connes used a Dixmier trace defined by a more restricted class of states involving Ces\`{a}ro means.  It was shown in \cite{LSS} that the (weaker) notion of measurable from Connes' smaller class of states involving Ces\`{a}ro means, the notion of measurable from dilation invariant states, or the notion of measurable from any larger class of generalised limits, were all equivalent, see (\cite{LSS}, \S 5.3), or (\cite{LSS}, Thm 6.6) in particular.

A.~Carey, J.~Phillips and the second author, by the content of \cite{CPS}, extended the formula (\ref{eq:res2}) 
to non-measurable operators (in the sense of Connes).  The results were a generalisation
to $\tau$-compact operators for von Neumann algebras with faithful normal semifinite trace $\tau$.  In this setting the $s$-numbers $\mu_s(T)$
of a $\tau$-compact operator, the generalisation of singular values, are continuous instead of discrete and one considers the Dixmier trace as an
expression $\tr_\upsilon(T) := \upsilon( \frac{1}{\log(1+t)} \int_1^t \mu_s(T)ds )$
for a dilation invariant state $\upsilon$ on  $L^\infty([1,\infty))$.  We phrase the extension of (\ref{eq:res2}) in the language of $B(H)$
(the bounded linear operator on $H$).
For $A \in B(H)$ and $0 < T \in \mathcal{L}^{1,\infty}$, define
$$
\zeta_{A,T}(s) := \Tr(AT^s),
$$
then (\cite{CPS}, Thm 3.8) states
\begin{equation} \label{eq:res1}
\tr_\upsilon(AT) = \tilde{\upsilon}\text{-}\lim_{s \to 1^+} (s-1)\zeta_{A,T}(s)
\end{equation}
for a `maximally invariant mean' (a dilation, power and Ces\`{a}ro invariant state) 
$\upsilon$ of $L^\infty([1,\infty))$.
The notation $\tilde{\upsilon}\text{-}\lim_{s \to 1^+} f(s)$ stands for $\tilde{\upsilon}(f(1+t^{-1}))$ where $\tilde{\upsilon}(f(t)) := \upsilon(f(\log t))$ for $f \in L^\infty([0,\infty))$, and $\tr_\upsilon(AT)$ stands for the linear extension of the weight
$\tr_\upsilon(\sqrt{A}T\sqrt{A})$, $A>0$.  Conditions on $\upsilon$ were reduced to dilation and power invariance in a later text of Carey, A.~Rennie, Sedaev and the second author, see (\cite{CRSS}, Thm 4.11).  However, (\cite{CRSS}, Thm 4.11) achieved (\ref{eq:res1}) \emph{only} for $A=1$.
  
In this note we show the utility of the noncommutative residue to the study of the noncommutative integral (taken in most texts on noncommutative geometry to be given by the lhs of (\ref{eq:res0}) or (\ref{eq:res1})).

Our first task will be to prove that
(\cite{CRSS}, Thm 4.11) can achieve the formula (\ref{eq:res1}) for $A \not=1$ with the same weakened conditions on the generalised limit $\upsilon$.  This is shown in Corollary \ref{cor:resAcont}.  
We also adapt the formula (\ref{eq:res1}) to
the class $\mathcal{L}(BL \cap DL)$ of `dilation and power invariant' states on $\ell^\infty$. The preliminaries will make the notation $\mathcal{L}(BL \cap DL)$ apparent.  This is done in Corollary \ref{cor:resA}.
The adaptation is important, since it shows that
the generalisations in \cite{CPS} and \cite{CRSS} to semifinite von Neumann algebras apply to the `original' type I
construction of Dixmier (used originally by Connes in \cite{C3}).  This step
is not entirely trivial.  There are subtle distinctions between Dixmier traces involving the discrete values
$\mu_n(T)$ and states on $\ell^\infty$ and those involving the function $\mu_{s} = \mu_{\floor{s}}(T)$ ($\floor{s}$ is the floor function)
and states on $L^\infty([1,\infty))$, even though they provide equivalent sets of traces, see (\cite{LSS}, Thm 6.2).

With the correspondence between noncommutative residues and
the noncommutative integral (the lhs of (\ref{eq:res0}) or (\ref{eq:res1})) firmly in hand we use residues to show two analytic results.

The first result is a structure result for the noncommutative integral.  
Assume the situation is non-trivial, i.e.~$0 < T \in \mathcal{L}^{1,\infty}$
with $\Tr_\omega(T) > 0$.
For $A \in B(H)$ set
\begin{equation} \label{eq:normint}
\phi_\omega(A) := \frac{\Tr_\omega(AT)}{\Tr_\omega(T)}.
\end{equation}
Note
$\phi_\omega$ is a state of $B(H)$.  Then Theorem \ref{thm:structure} says that, when $\omega \in \mathcal{L}(BL \cap DL)$,
\begin{equation} \label{eq:intchar}
\phi_\omega(A) = L_\omega( \inprod{h_m}{Ah_m} )
\end{equation}
where $\{ h_m \}_{m=1}^\infty$ is any complete orthonormal system
of eigenvectors for $T$ (\cite{S}, \S 1) and $L_\omega$ is a generalised{} limit.
The characterisation (\ref{eq:intchar}) shows,
when the sequence $\{ \inprod{h_m}{Ah_m} \}_{m=1}^\infty$ is
convergent at infinity, the state $\phi_\omega$ is \emph{uniquely} and completely characterised by the \emph{eigenvectors} of $T$.
The \emph{eigenvalues} of $T$ are linked solely to the scale factor $\Tr_\omega(T)$.  The flat $1$-torus and the noncommutative torus provide examples in the text.   This is a revealing insight.
Weyl's formula on the asymptotics of the eigenvalues of the Laplacian has been cited as the staring point of integration in noncommutative geometry, see for example (\cite{GBVF}, \S 7.6).  However,
the eigenvectors of the Laplacian, not the eigenvalues, turn out to be
the critical determinants of the value obtained by `integration'.

The second result we obtain is normality
criteria for the noncommutative integral.  While $\phi_\omega$ is a state of $B(H)$,
it is not a normal state.  This is problematic for an `integral', since monotone and dominated convergence cannot be applied. The normality of $\phi_\omega$ on proper weakly closed $^*$-subalgebras of $B(H)$ is an open question.
The characterisation (\ref{eq:intchar}) of $\phi_\omega$, for $\omega \in \mathcal{L}(BL \cap DL)$,
as a generalised limit of the states $A \mapsto \inprod{h_m}{Ah_m}$,
gives us valuable purchase.  
Let $\mathcal{M} \subset B(H)$ be a von Neumann algebra.
We say a positive compact operator
$T$ is $(\mathcal{M},h)$-dominated if there exists $h \in H$
such that $\nm{Ph_m} \leq \nm{Ph}$ for all projections $P \in \mathcal{M}$.
Think here of $H = L^2(F,\mathcal{B},\mu)$ for a $\sigma$-finite
measure space $(F,\mathcal{B},\mu)$ and $\mathcal{M}$ as multiplication operators of $L^\infty$-functions.  Then $T$ being $(\mathcal{M},h)$-dominated
is the same as $\int_J |h_m(x)|^2 d\mu(x)
\leq \int_J |h(x)|^2 d\mu(x)$ for all $J \subset \mathcal{B}$,
which is equivalent to the statement $|h_m|^2$ are
dominated by $|h|^2 \in L^1(F,\mathcal{B},\mu)$ $\mu$-a.e..  In Theorem
\ref{thm:normM} we show that if $T$ is $(\mathcal{M},h)$-dominated, then $\Tr_\omega( \cdot T) \in \mathcal{M}_*$.

\section{Preliminaries}

\subsection{Preliminaries on (Discrete) Dixmier Traces} \label{sec:prelim}

Let $\ceil{x}$, $x \geq 0$, denote the ceiling function.
Define the maps $\ell^\infty \to \ell^\infty$ for $j \in \NN$ by
\begin{eqnarray*}
T_j( \{a_k\}_{k=1}^\infty ) & = & \{ a_{k+j} \}_{k=1}^\infty \\
D_j( \{a_k \}_{k=1}^\infty) & = & \{ a_{\ceil{j^{-1}k}} \}_{k=1}^\infty .
\end{eqnarray*}
Set $BL := \inset{0 < \omega \in (\ell^{\infty})^*}{\omega(1) = 1, \omega \circ T_j = \omega \ \forall j \in \NN}$ (the set of Banach Limits)
and $DL := \inset{0 < \omega \in (\ell^{\infty})^*}{\omega(1) = 1, \omega \circ D_j = \omega \ \forall j \in \NN}$.
Both sets of states on $\ell^\infty$ satisfy
\begin{equation} \label{eq:genL}
\liminf_k a_k \leq \omega(\{ a_k \}_{k=1}^\infty) \leq \limsup_k a_k \ , a_k \geq 0 .
\end{equation}
Any state of $\ell^\infty$ satisfying (\ref{eq:genL}) is termed a generalised limit since it extends $\lim$ on $c$ to $\ell^\infty$.
Let $0 < T \in \mathcal{L}^{1,\infty}$.  Set $\gamma(T) := \left\{ \log(1+k)^{-1} \sum_{j=1}^k \mu_j(T)\right\}_{k=1}^\infty \in \ell^\infty$
and define $DL_2 := \inset{0 < \omega \in (\ell^{\infty})^*}{\omega(1) = 1,
\omega \text{ satisfies (\ref{eq:genL}), and }
\omega(D_2(\gamma(T))) = \omega(\gamma(T)) \fa 0 < T \in \mathcal{L}^{1,\infty}}$.  
From (\cite{CN}, pp.~304-305) or (\cite{AF}, Prop 5.2), for any $\omega \in DL_2$,
$$
\Tr_\omega(T) := \omega(\gamma(T))
$$
defines a tracial weight.  Denote by $\Tr_\omega$ as well the linear extension.  Then $\Tr_\omega$ is a finite trace on $\mathcal{L}^{1,\infty}$ that vanishes on the separable part $\mathcal{L}_0^{1,\infty}$.  The separable part $\mathcal{L}_0^{1,\infty}$ is the closure of finite rank operators in the norm $\nm{\cdot}_{1,\infty}$.
The condition that $\omega \in DL_2$ is weaker than the condition that
$\omega$ be dilation invariant, and weaker than Dixmier's original condition that
$\omega \in BL \cap DL$, \cite{Dix}.

\subsection{Preliminaries on (Continuous) Dixmier Traces} \label{sec:res.1}

For $a > 0$ define the maps $L^\infty([0,\infty)) \to L^\infty([0,\infty))$ by
\begin{eqnarray*}
T_a(f)(t) & = & f(t+a) \\
D_a(f)(t) & = & f(a^{-1}t) .
\end{eqnarray*}
Let $\phi$ be a state on $L^\infty([0,\infty))$ satisfying:
\begin{prop2list}{12}{0}{2}
\item $\phi(g) = 0$ for $g \in C_0([0,\infty))$;
\item $\text{ess.-}\liminf_{t \to \infty} g(t) \leq \phi(g) \leq \text{ess.-}\limsup_{t \to \infty} g(t)$
for $0 < g \in L^\infty([0,\infty))$;
\item $\phi(g) = \phi(T_a(g))$, $a > 0$, for $g \in L^\infty([0,\infty))$.
\end{prop2list}
Then $\phi$ is called a (continuous) Banach limit,
$\phi \in BL[0,\infty)$, (\cite{LSS}, \S 1.1), (\cite{CPS}, \S 1.2).  If $\phi$ satisfies
\begin{prop2list}{12}{0}{2}
\item[\textrm{(iii)'}] $\phi(g) = \phi(D_a(g))$, $a > 0$, for $g \in L^\infty([0,\infty))$
\end{prop2list}
instead of (iii), we denote this by $\phi \in DL[0,\infty)$.
Define $L^{-1} : L^\infty([1,\infty)) \to L^\infty([0,\infty))$ by
$$
L^{-1}(g)(t) = g(e^t)
$$
and
$$
L(\phi) := \phi \circ L^{-1} .
$$
It is known $T_{a} L^{-1} = L^{-1} D_{e^{-a}}$ and $D_{a^{-1}} L^{-1} = L^{-1} P^{a}$, $a \geq 1$, where
\begin{eqnarray*}
P^a(f)(t) & = & f(t^{a}),
\end{eqnarray*}
see (\cite{CPS}, \S 1.1).  Thus $L$ provides isometries
$BL[0,\infty) \to DL[1,\infty)$ and $DL[0,\infty) \to P[1,\infty)$ where the notations should be evident.

Let $0 < T \in \mathcal{L}^{1,\infty}$.  Set $\Gamma(T)(t) := \log(1+t)^{-1} \int_1^t \mu_s(T) ds$ where $\mu_s(T)$ are the $s$-numbers of $T$ relative to the canonical trace $\Tr$ on $B(H)$. Denote $L(\phi) \in DL_2[1,\infty)$ if $\phi$ satisfies (i) and (ii) above
and $L(\phi)(D_2(\Gamma(T))) = L(\phi)(\Gamma(T)) \fa 0 < T \in \mathcal{L}^{1,\infty}$.  
From (\cite{LSS}, \S 6) and (\cite{AF}, \S 5), for any $L(\phi) \in DL_2[1,\infty)$,
$$
\tr_{L(\phi)}(T) := L(\phi)(\Gamma(T))
$$
defines a tracial weight.  Denote by $\tr_{L(\phi)}$ as well the linear extension.  Then $\tr_{L(\phi)}$ is a finite trace on $\mathcal{L}^{1,\infty}$ that vanishes on the separable part $\mathcal{L}_0^{1,\infty}$.  It is evident if $\phi \in BL[0,\infty)$ then $L(\phi) \in DL_2[1,\infty)$.

\section{The Main Results}

We state the extension of (\cite{CRSS}, Thm 4.11).  For brevity we state the result only for $B(H)$.  The statement and proof for a general semifinite von Neumann algebra is apparent.  The proofs of Theorems \ref{thm:resPcont} and \ref{thm:resPA} and Corollaries
\ref{cor:resAcont} and \ref{cor:resA} are reserved for the technical section, Section \ref{sec:restech}.

\begin{thm} \label{thm:resPcont}
Let $P$ be a projection and $0 < T \in \mathcal{L}^{1,\infty}$.
For any $\phi \in BL[0,\infty) \cap DL[0,\infty)$,
$$
\tr_{L(\phi)}(PTP) = \phi \left( \frac{1}{r} \Tr(PT^{1+\frac{1}{r}}P) \right)  .
$$
Moreover, $\lim_{s \to 1^+} (s-1) \Tr(PT^{s}P)$ exists iff $PTP$ is measurable
and in either case 
$$
\tr_{\upsilon}(PTP) = \lim_{s \to 1^+} (s-1) \Tr(PT^{s}P) 
$$
for all $\upsilon \in DL_2[1,\infty)$.
\end{thm}

\begin{dfn}
Let $0 < T \in \mathcal{L}^{1,\infty}$.  We say $T$
is \emph{spectrally measurable} w.r.t.~$A \in B(H)$ (in the sense of Connes) if $PTP$ is measurable for all projections $P$ in the von Nuemann algebra generated by $A$ and $A^*$.
\end{dfn}

\begin{cor} \label{cor:resAcont}
Let $A \in B(H)$ and $0 < T \in \mathcal{L}^{1,\infty}$.
For any $\phi \in BL[0,\infty) \cap DL[0,\infty)$,
$$ 
\tr_{L(\phi)}(AT) = \phi \left( \frac 1r \zeta_{A,T}\left( 1+\frac{1}{r} \right) \right) .
$$
Moreover, $AT$ is measurable if $T$ is spectrally measurable w.r.t.~$A$
and
$$
\tr_{\upsilon}(AT) = \lim_{s \to 1^+} (s-1) \zeta_{A,T}(s) 
$$
for all $\upsilon \in DL_2[1,\infty)$.
\end{cor}

We now state the adaptation to `original' type I (discrete) Dixmier traces.
Define the averaging sequence $E : L^\infty([0,\infty)) \to \ell^\infty$ by
$$
E_k(f) :=  \int_{k-1}^k f(t)dt .
$$
Define the floor mapping 
$
p : \ell^\infty \to L^\infty([1,\infty))
$
by
$$
p( \{a_k \}_{k=1}^\infty)(t) :=
\sum_{k=1}^\infty  a_k \chi_{[k,k+1)}(t) .
$$
Define, finally, the isometry $\mathcal{L} : (\ell^{\infty})^*
\to (\ell^{\infty})^*$ by
$$
\mathcal{L}(\bl) := \bl \circ E \circ L^{-1} \circ p .
$$
Denote by $\mathcal{L}(BL \cap DL)$ the image of translation and dilation invariant states on $\ell^\infty$ under $\mathcal{L}$.
Unlike the continuous case, it is not evident that $\mathcal{L}(\xi) \in DL_2$ if $\xi \in BL \cap DL$.

\begin{thm} \label{thm:resPA}
Let $P$ be a projection and $0 < T \in \mathcal{L}^{1,\infty}$.
For any $\bl \in BL \cap DL$, $\mathcal{L}(\bl) \in DL_2$ and
$$
\Tr_{\mathcal{L}(\bl)}(PTP) = \bl \left( \frac 1k \Tr(PT^{1+\frac 1k}P) \right) .
$$
Moreover, $\lim_{k \to \infty} \frac 1k \Tr(PT^{1+\frac 1k}P)$ exists iff $PTP$ is measurable and in either case 
$$
\Tr_{\omega}(PTP) = \lim_{k \to \infty} \frac 1k \Tr(PT^{1+\frac 1k}P)
$$
for all $\omega \in DL_2$.
\end{thm}

\begin{cor} \label{cor:resA}
Let $A \in B(H)$ and $0 < T \in \mathcal{L}^{1,\infty}$.
For any $\bl \in BL \cap DL$,
$$
\Tr_{\mathcal{L}(\bl)}(AT) = \bl \left( \frac 1k \zeta_{A,T}\left( 1+\frac 1k \right) \right) .
$$
Moreover, $AT$ is measurable if $T$ is spectrally measurable w.r.t.~$A$
and
$$
\Tr_{\omega}(AT) = \lim_{k \to \infty} \frac 1k \zeta_{A,T}\left( 1+\frac 1k \right)
 = \lim_{s \to 1^+} (s-1) \zeta_{A,T}(s) $$
for all $\omega \in DL_2$.
\end{cor}

Here $AT$ measurable means $\Tr_{\omega}(AT)$ is independent of $\omega \in DL_2$.  In Corollary \ref{cor:resAcont} $AT$ measurable meant
$\tr_{\upsilon}(AT)$ independent of $\upsilon \in DL_2[1,\infty)$. 
Spectral measurability is sufficient for equivalence of the two notions when $A \not= P$ (equivalence when $A=P$ was shown in (\cite{LSS}, Cor 3.9)).  

\begin{rems} \label{rem:M_meas}
Let $PTP$ be measurable for all projections $P$ in a von Neumann algebra $\mathcal{M} \subset B(H)$.  It is clear $AT$ is (unambiguously) measurable for all $A \in \mathcal{M}$ since $\mathcal{M}$
contains the von Nuemann algebras generated by $A$ and $A^*$.
\end{rems}

We show two applications of residues.

\subsection{Structure of the noncommutative integral}

Let $0 < T \in \mathcal{L}^{1,\infty}$ be non-trivial,
i.e.~$\Tr_\omega(T) > 0$ for all $\omega \in DL_2$.
Following (\ref{eq:normint}) in the introduction,
\begin{equation*}
\phi_\omega(A) := \frac{\Tr_\omega(AT)}{\Tr_\omega(T)} \ , \ A \in B(H)
\end{equation*}
is a state of $B(H)$.

\begin{thm} \label{thm:structure}
Let $\omega \in \mathcal{L}(BL \cap DL)$.  There is a generalised{}
limit $L_\omega$
such that
$
\phi_\omega(A) = L_\omega \left( \inprod{h_m}{Ah_m} \right)
$
where $\{ h_m \}_{m=1}^\infty$ is any complete orthonormal system of eigenvectors of $T$.
\end{thm}

The proof is not overly technical.  We provide it here.
Take $\{h_m \}_{m=1}^\infty$ a complete orthonormal system of eigenvectors for $T$ such that $T h_m = \lambda_m h_m$.
Let $P_m$, $m \in \NN$, denote the one dimensional projections
onto $h_m$.  Define the map $\theta : \ell^\infty \to B(H)$ by
\begin{equation} \label{eq:gammamap}
\theta \left( \{ a_k\}_{k=1}^\infty \right) = \sum_{k=1}^\infty a_k P_k .
\end{equation}

\begin{lemma}
The map $\theta : \ell^\infty \to B(H)$ is an isometric injection
such that $\theta( \mathbf{1} ) = I$. Here $I$ is the identity of $B(H)$.
\end{lemma}
\begin{proof}
Set $\mathbf{a} = \{ a_k\}_{k=1}^\infty$.
Then $\theta(\mathbf{a})h
= \sum_{m=1}^\infty a_m \inprod{h_m}{h} h_m$
and $\nm{\theta(\mathbf{a})h}^2 
= \sum_{m=1}^\infty |a_m\inprod{h_m}{h}|^2
\leq \nm{\mathbf{a}}_\infty^2 
\sum_{m=1}^\infty |\inprod{h_m}{h}|^2
= \nm{\mathbf{a}}_\infty^2 \nm{h}^2$.
So $\theta(\mathbf{a})$ is a bounded linear operator with $\nm{\theta(\mathbf{a})} \leq \nm{\mathbf{a}}_\infty$.
Conversely $\theta(\mathbf{a})h_m = a_m h_m$ with
$\nm{\theta(\mathbf{a})h_m} = |a_m|$.  So
$\nm{\theta(\mathbf{a})} :=
\sup_{\| h \| \leq 1} \nm{\theta(\mathbf{a})h}
\geq \sup_{m} |a_m| =: \nm{\mathbf{a}}_\infty$.  This shows
$\theta$ is an isometry.  Finally, if $\theta(\mathbf{a}) = 0$,
$\theta(\mathbf{a})h_m = 0$ and hence $a_m = 0 \fa m \in \NN$.
Thus $\theta$ is injective.  It is evident
$\theta \left( \mathbf{1} \right) = \sum_{k=1}^\infty P_k = I$.
\end{proof}

Let $\omega \in DL_2$.  Define the linear functional $L_\omega : \ell^\infty \to \CC$ by
\begin{equation} \label{eq:Lmap}
L_\omega \left( \{ a_k\}_{k=1}^\infty \right) = \phi_\omega \left(
\theta( \{ a_k\}_{k=1}^\infty) \right) .
\end{equation}

\begin{lemma} \label{lemma:GL}
A state on $\ell^\infty$ vanishes on finite sequences if and only
if it is a generalised limit.
\end{lemma}
\begin{proof}
The if direction is evident.  Let $L$ be a state that vanishes
on finite sequences.
Let $\{a_k \}_{k=1}^\infty \geq 0$.
Set $e_N := (0,\ldots,0,1,\ldots)$ and $\alpha_N :=
(0,\ldots,0,a_N,a_{N+1},\ldots)$ where there are $N-1$ zeros.
Then $L(e_N) = L(1) = 1$ and $L(\alpha_N) = L(\{ a_k \}_{k=1}^\infty)$.
By positivity of $L$, 
$\inf_{k \geq N} a_k L(e_N) \leq L(\alpha_N) \leq \sup_{k \geq N} a_k L(e_N)$.  Hence $\inf_{k \geq N} a_k \leq L(\{a_k\}_{k=1}^\infty) \leq \sup_{k \geq N} a_k$
It follows $\liminf_{k} a_k \leq L(\{a_k\}_{k=1}^\infty) \leq \limsup_{k} a_k$ by taking $N \to \infty$.
\end{proof}

\begin{prop}
The map $L_\omega : \ell^\infty \to \CC$ is a generalised{}
limit.
\end{prop}
\begin{proof}
Without loss $\Tr_\omega(T) =1$.  It is evident $L:=L_\omega$ is positive and $L(\mathbf{1}) = 1$.  Hence $L$ is a state of $\ell^\infty$. 
Suppose $\{a_k\}_{k=1}^\infty \geq 0$.  For $N \geq 2$, $\theta(\{a_k\}_{k=1}^{N-1})T$ is finite rank and
$L(\{a_k\}_{k=1}^{N-1}) = \Tr_\omega(\theta(\{a_k\}_{k=1}^{N-1})T) = 0$. Apply the previous lemma.
\end{proof}

\subsubsection{Proof of Theorem \ref{thm:structure}}

Without loss $\Tr_\omega(T)=1$. Let $\bl \in BL \cap DL$. Using Corollary \ref{cor:resA},
\begin{eqnarray}
\phi_{\mathcal{L}(\bl)}(A)
& = & \bl( p^{-1} \Tr(AT^{1+p^{-1}}) ) \nonumber \\
& = & \bl( p^{-1} \sum_{m=1}^\infty \inprod{h_m}{AT^{1+p^{-1}}h_m}) \nonumber \\
& = & \bl( p^{-1} \sum_{m=1}^\infty \lambda_m^{1+p^{-1}} \inprod{h_m}{Ah_m}) . \label{eq:structureA}
\end{eqnarray}
Conversely, from (\ref{eq:Lmap}), 
$L_{\mathcal{L}(\bl)}(\{ \inprod{h_k}{Ah_k} \}_{k=1}^\infty) = \phi_{\mathcal{L}(\bl)}(\theta(\{ \inprod{h_k}{Ah_k} \}_{k=1}^\infty))$, and
\begin{eqnarray}
\phi_{\mathcal{L}(\bl)}(\theta(\{ \inprod{h_k}{Ah_k} \}_{k=1}^\infty))
& = & \bl( p^{-1} \Tr(\sum_{k=1}^\infty \inprod{h_k}{Ah_k} P_k T^{1+p^{-1}}) ) \nonumber \\
& = & \bl( p^{-1} \sum_{m=1}^\infty \inprod{h_m}{\sum_{k=1}^\infty \inprod{h_k}{Ah_k} P_k T^{1+p^{-1}}h_m}) \nonumber \\
& = & \bl( p^{-1} \sum_{m=1}^\infty \lambda_m^{1+p^{-1}} \inprod{h_m}{Ah_m}) . \label{eq:structureB}
\end{eqnarray}
Comparing (\ref{eq:structureA}) and (\ref{eq:structureB})
yields the result. \qed

\begin{ex} \label{example:sec3_1}
\textbf{1.}~Consider the Laplacian $\Delta = - d^2/d\theta^2$ on the flat 1-torus $\TT$.
From Theorem \ref{thm:structure}, $\phi_\omega(A) = L_\omega(\inprod{f_m}{Af_m})$
where $f_m(\theta) = e^{im\theta}$, $m \in \ZZ \cong \{0,1,-1,2,-2,\ldots\}$, for any $A \in B(L^2(\TT))$.
If $M_f$ is the multiplier of $f \in L^\infty(\TT)$ on $L^2(\TT)$,
$\inprod{f_m}{M_f f_m} = \int_{-\pi}^\pi f(\theta) d\theta$, and
$\Tr_\omega(M_f\Delta^{-1/2}) = 2 \phi_\omega(M_f) = 2 \int_{-\pi}^\pi f(\theta) d\theta$.  See \cite{LDS} for the equivalent statement for any (closed) compact Riemannian manifold.

\textbf{2.}~Consider two unitaries $u,v$ such that $uv = \lambda vu$,
for $\lambda := e^{2\pi i \theta} \in \SB$ (the unit circle). Denote by
$F_\theta(u,v)$ the $^*$-algebra of linear combinations
$\sum_{(m,n) \in J} a_{m,n} u^mv^n$, $J \subset \ZZ^2$ is a finite set, with product $ab = \sum_{r,s}( \sum_{m,n} a_{r-m,n} \lambda^{mn} b_{m,s-n} ) u^rv^s$
and involution $a^* = \sum_{r,s} (\lambda^{rs} \overline{a}_{-r,-s}) u^rv^s$, $a,b \in F_\theta(u,v)$. The assignment $\tau_0(a) = a_{0,0}$ is a faithful trace on $F_\theta(u,v)$.  Let $(H_\theta,\pi_\theta)$ denote the cyclic representation associated to $\tau_0$.
The closure, $C_\theta(u,v)$, of $\pi_\theta(F_\theta(u,v))$ in the operator norm is called a rotation C$^*$-algebra, \cite{Rieff1},
or the noncommutative torus ($\lambda \not=1$), (\cite{CN}, \S III.2.$\beta$ IV.6.$\alpha$  VI.3.c) (\cite{C4}, p.~166), (\cite{GBVF}, \S 12.2).   Canonically, finite linear combinations of
$u^mv^n \hookrightarrow H_\theta$ are dense in $H_\theta$.
Define $\Delta_\theta(u^mv^n) = (m^2 + n^2)u^mv^n$.
It can be shown that the `noncommutative laplacian' $\Delta_\theta$ has a unique positive extension (also denoted $\Delta_\theta$)
$\Delta_\theta : \Dom(\Delta_\theta) \to H_\theta$
with compact resolvent, see \textit{op.~cit.}.  The eigenvectors $h_{m,n} = u^mv^n \in H_\theta$ form a complete orthonormal system.
Note that
\begin{eqnarray*}
\inprod{h_{m,n}}{\pi_\theta(a)h_{m,n}} & = &
\sum_{k,l} \lambda^{-kl}\delta_{m,n}(k,-l) \lambda^{kl} \sum_{r,s} a_{k-r,s} \lambda^{rs} \delta_{m,n}(r,-l-s) \\
& = &  \sum_{k,l} \delta_{m,n}(k,l) a_{k-m,l-n} \lambda^{m(l-n)} = a_{0,0}
\end{eqnarray*}
for any $(m,n) \in \ZZ^2$.
Using the Cantor enumeration of $\ZZ^2$, it follows from Theorem
\ref{thm:structure} that
$
\Tr_\omega(\pi_\theta(a)\Delta_\theta^{-1}) = \pi \inprod{h_{0,0}}{\pi_\theta(a)h_{0,0}} = \pi \tau_0(a)
$, $\! \! \fa a \in F_\theta(u,v)$ ($\Tr_\omega(\Delta_\theta^{-1})=\pi$).  By uniform continuity
the same result follows for $C_\theta(u,v)$.
Thus Theorem \ref{thm:structure} provides a short proof of the known facts that $A\Delta_\theta^{-1}$, $A \in C_\theta(u,v)$, is measurable (in the sense of Connes) and $\Tr_\omega( \cdot \Delta_\theta^{-1})$
is a faithful trace on $C_\theta(u,v)$.
\end{ex}

\subsection{Conditions for normality of the noncommutative integral}

Let $\mathcal{M}$ be a weakly closed $^*$-subalgebra of $B(H)$.

\begin{dfn} \label{ref:dfn_Mdom}
A positive compact operator
$T$ is $(\mathcal{M},h)$-dominated if, for some complete orthonormal system $\{h_m\}_{m=1}^\infty$ of eigenvectors of $T$,
there exists $h \in H$ such that $\nm{Ph_m} \leq \nm{Ph}$ for all projections $P \in \mathcal{M}$.
\end{dfn}

\begin{thm} \label{thm:normM}
Let $0 < T \in \mathcal{L}^{1,\infty}$ be $(\mathcal{M},h)$-dominated.  Then $\Tr_\omega( \cdot T) \in \mathcal{M}_*$ for all $\omega \in \mathcal{L}(BL \cap DL)$.
\end{thm}
\begin{proof}
By hypothesis $\inprod{h_m}{Ph_m} \leq \inprod{h}{Ph}$ for all projections $P \in \mathcal{M}$.  Then $\inprod{h_m}{Ah_m} \leq \inprod{h}{Ah}$, $0 < A \in \mathcal{M}$, as $A$
is a uniform limit of finite linear positive spans of projections
(\cite{Ped}, \S 2.2.6 p.~23).  For any generalised limit
$L$ and $0 < A \in \mathcal{M}$,
\begin{equation} \label{eq:dom}
L(\inprod{h_m}{Ah_m})
\leq \limsup_{m \to \infty} \inprod{h_m}{Ah_m}
\leq \inprod{h}{Ah}.
\end{equation}
Let $\{ A_\alpha \}$ be a net of monotonically increasing positive elements of $\mathcal{M}$ with upper bound.  It follows that $\{ A_\alpha \}$ converges strongly
to a l.u.b.~$A \in \mathcal{M}$ (\cite{BR}, Lemma 2.4.19 p.~76).  From (\ref{eq:dom})
$
L(\inprod{h_m}{(A - A_\alpha)h_m})
\leq \inprod{h}{(A-A_\alpha)h}.
$
Since $\inprod{h}{(A-A_\alpha)h} \stackrel{\alpha}{\to} 0$,
$L(\inprod{h_m}{Ah_m}) = \sup_{\alpha} L(\inprod{h_m}{A_\alpha h_m})$.  From Theorem \ref{thm:structure}
$\phi_\omega(A) = \sup_\alpha \phi_\omega(A_\alpha)$
and $\phi_\omega$ is normal on $\mathcal{M}$ (\cite{Ped}, \S 3.6.1) (\cite{BR}, p.~76).  The result follows as $\Tr_\omega(\cdot T)
= \phi_\omega(\cdot) \Tr_\omega(T)$ (a scalar multiple of $\phi_\omega$).
\end{proof}

\begin{ex}
\textbf{1.}~Let $(F,\mu)$ be a $\sigma$-finite measure space.  Take $H=L^2(F,\mu)$
and $\mathcal{M} = L^\infty(F,\mu)$ acting by multipliers on $H$.  Let $T$ be any positive compact operator (or positive operator with compact resolvent) with
eigenfunctions $f_m$ satisfying $|f_m|^2 \leq g \in L^1(F,\mu)$ $\mu$-a.e..
Then $T$ is $(\mathcal{M},g)$-dominated.  For example, the eigenfunctions $f_m(\theta) = e^{im\theta}$ of
the Laplacian $\Delta$ on the 1-torus $\TT$ satisfy
$|f_m|^2 = 1 \in L^1(\TT)$.

\textbf{2.}~Let $\Delta_{\theta}$ be the `noncommutative laplacian' from Example \ref{example:sec3_1}.2.  From the example $\nm{Ph_{m,n}} = \nm{Ph_{0,0}}$ for all projections $P \in C_{\theta}(u,v)''$.  Hence $\Tr_\omega( \cdot \Delta_\theta^{-1})$ is a faithful normal trace
on $C_{\theta}(u,v)''$.
\end{ex}

\section{Technical Results} \label{sec:restech}

%
%

\subsection{Proof of Theorem \ref{thm:resPcont} and Corollary \ref{cor:resAcont}}

\begin{lemma} \label{prop:projcont}
Let $P$ be a projection and $0 < T \in \mathcal{L}^{1,\infty}$.
For any $\phi \in BL[0,\infty)$,
$$
\phi \left( \frac 1r \Tr(PT^{1+\frac 1r}P) \right)
= \phi \left( \frac 1r \Tr((PTP)^{1+\frac 1r}) \right).
$$
If either function is convergent at infinity $\phi$ can be replaced by $\lim$.
\end{lemma}
\begin{proof}
By (\cite{CPS}, Prop 3.6), $\phi \left( 1/r \Tr(PT^{1+1/r}) \right)
= \phi \left( 1/r \Tr((\sqrt{P}T\sqrt{P})^{1+1/r}) \right)$
for $\phi \in BL[0,\infty)$, and, if either function is convergent at infinity, $\phi$ can be replaced by $\lim$.
Clearly $1/r \Tr(PT^{1+1/r}P) = 1/r \Tr(PT^{1+1/r})$
 and $1/r \Tr((\sqrt{P}T\sqrt{P})^{1+1/r})
 = 1/r \Tr((PTP)^{1+1/r})$.
\end{proof}

\subsubsection{Proof of Theorem \ref{thm:resPcont}}

From (\cite{CRSS}, Thm 4.11)
$\tr_{L(\phi)}(V) = \phi \left( 1/r \Tr(V^{1+1/r}) \right)$
where $0 < V = PTP \in \mathcal{L}^{1,\infty}$.  An application of Lemma \ref{prop:projcont} yields the first display of Theorem \ref{thm:resPcont}.  Note $V$ is measurable iff
$V$ is Tauberian by (\cite{LSS}, Cor 3.9).  From formula (\ref{eq:res2}) $V$ is Tauberian iff the residue of $\zeta_V$ exists at $s=1$. From Lemma \ref{prop:projcont} the residue exists iff
$\lim_{r \to \infty} 1/r \Tr(PT^{1+1/r}P)$ exists. \qed

\subsubsection{Proof of Corollary \ref{cor:resAcont}}

Let $A \in B(H)$ and let $\mathcal{M}(A)$ denote the von Neumann algebra generated by $A$ and $A^*$.  Note that $A$ is the 
uniform limit of finite linear spans of projections in $\mathcal{M}(A)$ (\cite{Ped}, \S 2.2.6 p.~23).   Note also $\tr_{L(\phi)}(\cdot T)$
and $\phi \left( 1/r \Tr( \cdot T^{1+1/r}) \right)$ are positive linear functionals on $B(H)$ and so uniformly continuous
(\cite{BR}, Prop 2.3.11 p.~49).
Hence there is a finite set of scalars $c_{j,N} \in \CC$
and projections $P_{j,N} \in \mathcal{M}(A)$, $N \in \NN$, such that
\begin{eqnarray*}
\tr_{L(\phi)}(AT) & = & \lim_{N \to \infty} \sum_{j} c_{j,N} \tr_{L(\phi)}(P_{j,N}TP_{j,N}) \\
& \stackrel{\mathrm{(Thm \, \ref{thm:resPcont})}}{=} & \lim_{N \to \infty}  \sum_{j} c_{j,N} \phi \left( \frac 1r \Tr( P_{j,N} T^{1+\frac 1r} P_{j,N} ) \right) \\
& = & \phi \left( \frac 1r \Tr(AT^{1+\frac 1r}) \right) .
\end{eqnarray*}
If $T$ is spectrally measurable w.r.t~$A$, notice that
$| \lim_{r \to \infty} 1/r \Tr(AT^{1+1/r}) | = \lim_{r \to \infty} 1/r | \Tr(APT^{1+1/r}P) | \leq \nm{A} \lim_{r \to \infty} 1/r \Tr(PT^{1+1/r}P)$. Here $P$ is the maximal projection in $\mathcal{M}(A)$, see (\cite{KR1}, p.~309).  Hence
$\lim_{r \to \infty} 1/r \Tr( \cdot T^{1+1/r})$ is uniformly continuous on $\mathcal{M}(A)$.
For each $\upsilon \in DL_2[1,\infty)$
\begin{eqnarray*}
\tr_{\upsilon}(AT) & = & \lim_{N \to \infty} \sum_{j} c_{j,N} \tr_{\upsilon}(P_{j,N}TP_{j,N}) \\
& \stackrel{\mathrm{(Thm \, \ref{thm:resPcont})}}{=} & \lim_{N \to \infty} c_{j,N} \sum_{j} \lim_{r \to \infty} \frac 1r \Tr( P_{j,N} T^{1 + \frac 1r} P_{j,N} )  \\
& = & \lim_{r \to \infty} \frac{1}{r} \Tr(AT^{1+\frac 1r}) .
\end{eqnarray*}
The value $\tr_{\upsilon}(AT)$ is independent of $\upsilon$, so $AT$ is measurable.

\subsection{Proof of Theorem \ref{thm:resPA} and Corollary \ref{cor:resA}}

Let $f \in L^\infty([0,\infty))$ be an everywhere defined function
of the form $f(t) = \frac{g(t)}{t}$ where $g$ is increasing.  For $b > a > 0$
we note the trivial fact
$$
\sup_{t \in [a,b]} f(t) - \inf_{t \in [a,b]} f(t) \leq \frac{g(b)}{a}
- \frac{g(a)}{b}.
$$
Throughout this section $\bl$ is a state on $\ell^\infty$.
For brevity $\bl(a_k)$ denotes $\bl(\{a_k\}_{k=1}^\infty)$.
Also $\bl(h(ak-b))$,
for $h \in L^\infty([0,\infty))$, $a,b >0$, denotes
$\bl( \{ h(ak-b) \}_{k=\ceil{b/a}}^\infty)$.

\begin{lemma} \label{lemma:tech0_1}
Let $\bl$ be $T_1$-invariant. Then $\bl$ is $T_j$-invariant for
all $j \in \NN$ and $\bl(f(k+a)) = \bl(f(k))$, $a > 0$.
\end{lemma}
\begin{proof}
That $T_1$-invariance implies $T_j$-invariance, $j \in \NN$, is evident from induction.  For $a$ not a natural number
choose $j \in \NN$ such that $j-1 < a < j$.  Then,
\begin{eqnarray*}
\bl(|f(k+a)-f(k+j)|) & \leq & \bl ( \sup_{t \in [k+j-1,k+j]} f(t)
- \inf_{t \in [k+j-1,k+j]} f(t) ) \\
& \leq & \bl( \frac{g(k+j)}{k+j-1})  - \bl(\frac{g(k+j-1)}{k+j}) \\
& = & \bl( \frac{g(k)}{k-1})  - \bl(\frac{g(k)}{k+1}) \\
& = & \bl(\frac{2k}{(k+1)(k-1)} \frac{g(k)}{k}) \\
& \leq & \nm{f}_{\infty} \bl(\frac{2k}{(k-1)(k+1)}) = 0 \\
\end{eqnarray*}
since $\bl$ vanishes on $c_0$.  Hence
$\bl(f(k+a)) = \bl(f(k+j)) = \bl(f(k))$.
\end{proof}

\begin{lemma} \label{lemma:tech0_2}
Let $\xi$ be $T_1$-invariant.  For any $a > 0$, $b \geq 0$,
$$
\xi \left( \sup_{t \in [ak-b,ak+b]} f(t)- \inf_{t \in [ak-b,ak+b]} f(t) \right) = 0 .
$$
\end{lemma}
\begin{proof}
Let $M_k := \sup_{t \in [ak-b,ak+b]} f(t)- \inf_{t \in [ak-b,ak+b]} f(t)$ for $k \geq \ceil{b/a}$.  Repeating the steps of the previous lemma,
\begin{eqnarray*}
\xi(M_k) & \leq & \xi(\frac{g(ak+b)}{ak-b}) - \xi(\frac{g(ak-b)}{ak+b}) \\
& = & \xi(\frac{g(ak+b)}{ak-b})) - \xi(f_1(k))  \\
& \stackrel{(*)}{=} & \xi(\frac{g(ak+b)}{ak-b})) - \xi(f_1(k+2b/a)) \\
& = & \xi(\frac{g(ak+b)}{ak+b} \frac{4b(ak+b)}{(ak-b)(ak+3b)}) \\
& \leq & \nm{f}_{\infty}\xi(\frac{4b(ak+b)}{(ak-b)(ak+3b)}) = 0 .
\end{eqnarray*}
At (*) Lemma \ref{lemma:tech0_1} was applied 
to the function $f_{1}(t) := g(at - b)/(at+b) = (g(at-b)(at+b)^{-1}t)/t$, $t \geq b/a$, $0$ otherwise.
\end{proof}

\begin{lemma} \label{lemma:tech0_3}  
Let $\bl$ be $T_1$-invariant and $D_j$-invariant, $j \in \NN$.
For any $a> 0$, $\bl(f(\frac{k}{a})) = \bl(f(k))$.
\end{lemma}
\begin{proof}
For $j \in \NN$, $\bl(|f(\frac kj) - f(\ceil{\frac kj})|)
\leq \bl \left( \sup_{t \in [\frac{k}{j}-1,\frac{k}{j}+1]} f(t)- \inf_{t \in [\frac{k}{j}-1,\frac{k}{j}+1]} f(t) \right) = 0$ by
the previous lemma.  Hence
$\bl(f(\frac kj)) = \bl(f(\ceil{\frac kj})) = \bl(f(k))$ (*).
Take $p \in \NN$.
By applying the previous lemma and (*) to the function
$f_1(t) = (p^{-1}g(pt))/t = f(pt)$, it follows
$\bl(f(\frac{p}{j}k)) = \bl(f(p\frac{k}{j}))
= \bl(f(pk)) = \bl(f(p\frac{k}{p})) = \bl(f(k))$ (\dag).
Choose integers $p,j$ such that
$\frac{p}{j+1} \leq \frac 1a \leq \frac{p}{j}$.
Now $|f(\frac{k}{a}) - f(\frac{p}{j}k)|
\leq \sup_{t \in [\frac{p}{j+1}k, \frac{p}{j}k]} f(t)
- \inf_{t \in [\frac{p}{j+1}k, \frac{p}{j}k]} f(t)$.
Hence
\begin{eqnarray*}
\bl(|f(\frac{k}{a}) - f(\frac{p}{j}k)|) & \leq & \bl(\frac{j+1}{pk} g(\frac{p}{j}k)) - 
\bl(\frac{j}{pk}g(\frac{p}{j+1}k))  \\
& = & \frac{j+1}{j}\bl(f(\frac{p}{j}k)) - 
\frac{j}{j+1} \bl(f(\frac{p}{j+1}k))  \\
& \stackrel{\mathrm{by \, }(\dag)}{=} & (\frac{j+1}{j}-\frac{j}{j+1}) \bl(f(k)) \\
& \leq & \nm{f}_{\infty} \frac{2j+1}{j(j+1)} .
\end{eqnarray*}
Without loss, by adjusting $p$ proportionately, $j$ can be chosen arbitrarily large.
Hence $\bl(f(\frac{k}{a})) = \lim_{j \to \infty} \bl(f(\frac{p}{j}k)) = \bl(f(k))$ by (\dag).
\end{proof}

Define the averaging sequence $E : L^\infty([0,\infty)) \to \ell^\infty$ by
$
E_k(f) :=  \int_{k-1}^k f(t)dt
$.
For $a>0$, $b \geq0$, we abuse notation and write
$E_{ak+b}(f) :=  \int_{ak+b-1}^{ak+b} f(t)dt$.

\begin{lemma} \label{lemma:tech0_4}
Let $\bl$ be $T_1$-invariant.  For $a > 0$, $b \geq 0$,
$\bl(E_{ak+b}(f)) = \bl(f(ak+b))$.
\end{lemma}
\begin{proof}
Let $c = b+1$. Then $\inf_{t \in [ak-c,ak+c]} f(t)
\leq E_{ak+b}(f) \leq \sup_{t \in [ak-c,ak+c]} f(t)$.
Hence $|f(ak+b) - E_{ak+b}(f)| \leq \sup_{t \in [ak-c,ak+c]} f(t)- \inf_{t \in [ak-c,ak+c]} f(t)$.  The result follows by Lemma \ref{lemma:tech0_2}.
\end{proof}

\begin{lemma} \label{lemma:important}
Let $\bl$ be $T_1$-invariant and $D_j$-invariant, $j \in \NN$.
For any $a > 0$,
$$
\bl(f(k)) = \bl(E_k(T_a(f))) = \bl(E_k(D_a(f))) .
$$
\end{lemma}
\begin{proof}
That $\bl(E_k(T_a(f))) = \bl(E_{k+a}(f)) = \bl(f(k+a)) = \bl(f(k))$ is immediate with Lemma \ref{lemma:tech0_4}.  It also follows from (\cite{LSS}, Lemma 2.10).
Using the substitution $\frac ta \to t$,
$$
\bl(E_k(D_a(f))) := \bl( \int_{k-1}^k f(\frac{t}{a}) dt)
= \bl( a \int_{\frac ka -\frac 1a }^{\frac ka} f(t) dt).
$$
We have the equality $\bl( a \int_{\frac ka -\frac 1a }^{\frac ka} f(t) dt) = \bl( \int_{\frac ka - 1 }^{\frac ka} f(t) dt)$
from Lemma \ref{lemma:tech0_2} since
\begin{multline*}
\left| \int_{\frac ka -1}^{\frac ka} f(t)dt - a \int_{\frac ka -\frac 1a }^{\frac ka} f(t) dt \right| \\
 \leq \sup_{t \in [\frac{k}{a}-1,\frac{k}{a}]} f(t) + \sup_{t \in [\frac{k}{a}- \frac 1a ,\frac{k}{a}]} f(t)
 - (\inf_{t \in [\frac{k}{a}-1,\frac{k}{a}]} f(t) + \inf_{t \in [\frac{k}{a}- \frac 1a ,\frac{k}{a}]} f(t) )  \\
 \leq 2(\sup_{t \in [\frac{k}{a}-1,\frac{k}{a}]} f(t) - \inf_{t \in [\frac{k}{a}-1,\frac{k}{a}]} f(t)) .
\end{multline*}
From
$$
\bl( \int_{\frac ka - 1 }^{\frac ka} f(t) dt) = \bl(E_{\frac{k}{a}}(f)) \stackrel{\mathrm{(Lemma \, \ref{lemma:tech0_4})}}{=}  \bl(f(\frac ka)) \stackrel{\mathrm{(Lemma \, \ref{lemma:tech0_3})}}{=}  \bl(f(k))
$$
we obtain $\bl(E_k(D_a(f)))= \bl(f(k))$.
\end{proof}

Let $\bl \in BL \cap DL$.  It then follows, see (\cite{LSS}, Lemma 2.10) for example, that $\phi := \bl \circ E \in BL[0,\infty)$.  With Lemma \ref{lemma:important} we have, in addition, the property 
\begin{prop2list}{12}{0}{2}
\item[\text(iv)] $\phi(f) = \phi(D_a(f))$, $a > 0$, where
$f(t) := \frac 1t \Tr(QV^{1+ \frac 1t}Q)$, $t \geq 1$ ($0$ otherwise), $0 < V \in \mathcal{L}^{1,\infty}$, $Q$ a projection.
\end{prop2list}
Note that $L(\phi) = \phi \circ L^{-1} = \bl \circ E \circ L^{-1}$
belongs to $DL[1,\infty)$ and satisfies
\begin{prop2list}{12}{0}{2}
\item[\text{(iv')}] $\phi \circ L^{-1} (g) = \phi \circ L^{-1}(P^a(g))$, $a > 0$, where
$g(t) := \frac 1{\ln(1 + t)} \int_1^t \mu_{\floor{s}}(QVQ) ds$, $t \geq 1$ ($0$ otherwise), $0 < V \in \mathcal{L}^{1,\infty}$, $Q$ a projection.
\end{prop2list}
Property (iv') follows from Lemma \ref{lemma:important} by noting
that we have $\phi \circ L^{-1}(P^a(g)) = \bl (E_k(D_{a^{-1}}(L^{-1}(g))))$.  In particular,
$L^{-1}(g)$ has the form
$$
g(e^t) := \frac{1}{\ln(1+e^t)} \int_1^{e^t} \mu_{\floor{s}}(QVQ) ds,  t \geq 0
$$
for $0 < V \in \mathcal{L}^{1,\infty}$, $Q$ a projection.  This is equivalent to using the function
$$
g(e^t) = \frac{h(t)}{t}
$$
where $h(t) = \int_1^{e^t} \mu_{\floor{s}}(QVQ) ds$ is increasing.

\begin{rems} \label{rem:BLsuff}
Let $\phi \in BL[0,\infty)$ satisfy (iv) and (iv')
(and so $\phi$ is a state of $L^\infty([0,\infty))$
satisfying conditions (i),(ii),(iii),(iv),(iv')).
From an inspection of \cite{CRSS} and \cite{CPS},
these conditions are sufficient to repeat the conclusion of
(\cite{CRSS}, Thm 4.11).  Let us briefly mention why.
The requirement for $D_2$- and $P^\alpha$-invariance, $\alpha > 1$,
of $L(\phi)$ in the proof of (\cite{CRSS}, Thm 4.11) occurred
in three places.  Firstly the application of the
weak$^*$-Karamata theorem, then (\cite{CRSS}, Prop 4.3)
and (\cite{CRSS}, Cor 4.4).  Condition (iv') is exactly
what is required in the last display of (\cite{CRSS}, p.~264),
which is the only place $P^\alpha$-invariance is used
for (\cite{CRSS}, Prop 4.3). Hence (\cite{CRSS}, Prop 4.3) is true under condition (iv').  (\cite{CRSS}, Cor 4.4)
follows from (\cite{CRSS}, Prop 4.3). The property
of $D_2$-invariance is not an issue for $L(\phi)$ since
it is completely dilation invariant by $\phi \in BL[0,\infty)$.

What is left is weak$^*$-Karamata, i.e.~to achieve the last display on (\cite{CRSS}, p.~271).  In (\cite{CPS}, Thm 2.2), take the special choice of $h_T(t) = \Tr(T^{1+1/r}) = r f(r)$ where $f$ is in (iv) ($Q=1$, $V=T$), $0 < T \in \mathcal{L}^{1,\infty}$.  Dilation invariance
is used in the proof of (\cite{CPS}, Thm 2.2)
on the last display of (\cite{CPS}, p.~77).
Indeed, for our special choice of $h_T$, using the
notation of $\beta$ and $C$ from \cite{CPS},
$
\phi(1/r \int_0^\infty e^{-t/(r/(n+1))} d\beta(t))
   = 1/(n+1) \phi(1/(r/(n+1)) \Tr(T^{1+1/(r/(n+1))}))
   =  1/(n+1) \phi(1/r \Tr(T^{1+1/r}))
   = C/(n+1) .
$
The second equality is exactly (iv).
So the last display of (\cite{CPS}, p.~77) holds.
The rest of the argument of (\cite{CPS}, Thm 2.2)
carries through and with its result we obtain
the last display on (\cite{CRSS}, p.~271).  The rest of the argument
of (\cite{CRSS}, Thm 4.11) now carries through.
\end{rems}

Define the floor mapping 
$
p : \ell^\infty \to L^\infty([1,\infty))
$
by
$$
p( \{a_k \}_{k=1}^\infty)(t) :=
\sum_{k=1}^\infty a_{k} \chi_{[k,k+1)}(t)
$$
and the restriction mapping
$
r : B([1,\infty)) \to \ell^\infty
$
for everywhere defined bounded functions
by
$$
r(f) =  \{ f(k) \}_{k=1}^\infty .
$$

\begin{lemma}  \label{lemma:tech0_6}
For $\bl \in BL$ set $\phi:= \bl \circ E$.  Then
$$
\lim_{n \to \infty} \left(\sup_{t \in [n,n+1)} g(t) - \inf_{t \in [n,n+1)} g(t) \right) = 0
$$
and
$$
\phi \circ L^{-1}(|g(t) - pr(g)(t)|) = 0 .
$$
\end{lemma}
\begin{proof}
Let $M_n = \sup_{t \in [n,n+1)} g(t) - \inf_{t \in [n,n+1)} g(t)$.  Then
\begin{eqnarray*}
M_n & \leq & \frac{1}{\ln(1+n)} \int_1^{n+1} \mu_{\floor{s}}(QVQ) ds
- \frac{1}{\ln(1+n+1)} \int_1^{n} \mu_{\floor{s}}(QVQ) ds \\
& = & \left( 1- \frac{\ln(1+n)}{\ln(2+n)} \right) \frac{1}{\ln(1+n)}
\int_1^{n} \mu_{\floor{s}}(QVQ) ds \\
& & \hspace*{4.8cm} + \ \frac{1}{\ln(1+n)} \int_{n}^{n+1} \mu_{\floor{s}}(QVQ) ds \\
& \leq & \left( 1- \frac{\ln(1+n)}{\ln(2+n)} \right) \nm{QVQ}_{1,\infty}
+ \ln(1+n)^{-1} \mu_{n}(QVQ).
\end{eqnarray*}
Hence $\lim_n M_n = 0$.  Now,
$$
\sup_{t \in [n,n+1)} |g(t) - pr(g)(t)| \leq 
\sup_{t \in [n,n+1)} |g(t) - g(n)| \leq M_n .
$$
Consequently $|g(t) - pr(g)(t)| \leq \sum_{n=1}^\infty M_n \chi_{[n,n+1)}(t)$ and we have $\limsup_{t \to \infty} |g(t) - pr(g)(t)| \leq \limsup_{n \to \infty} M_n = 0$.  It follows
$\phi \circ L^{-1}(|g(t) - pr(g)(t)|) = 0$.
\end{proof}

Define the mapping $\mathcal{L} : (\ell^{\infty})^*
\to (\ell^{\infty})^*$ by
$$
\mathcal{L}(\bl) := \bl \circ E \circ L^{-1} \circ p .
$$

\subsubsection{Proof of Theorem \ref{thm:resPA}}

\begin{prop} \label{prop:trace}
Let $\bl \in BL$.  Then $\mathcal{L}(\bl) \in DL_2$.
\end{prop}
\begin{proof}
It is clear $\mathcal{L}(\bl)(1) = 1$ and $\mathcal{L}$ is positivity preserving.  Hence $\mathcal{L}(\bl)$ is a state on $\ell^\infty$.   If $\{b_k\}$ is a finite sequence,
$E \circ L^{-1} \circ p(\{b_k\})$ is a finite sequence.
Hence $\mathcal{L}(\bl)(\{ b_k \}) = \bl(E \circ L^{-1} \circ p(\{b_n\})) = 0$.  From Lemma \ref{lemma:GL}, (\ref{eq:genL}) is fulfilled.
Set $\gamma := r(g) = \{ \ln(1+k)^{-1}\sum_{j=1}^k \mu_{j}(T) \}_{k=1}^\infty$.  Notice
\begin{equation*}
(D_2p - pD_2)(\{ a_k \}_{k=1}^\infty)  =  \sum_{k=1}^\infty (a_{k+1} -a_k)\chi_{[2k+1,2k+2)}(t) .
\end{equation*}
So
\begin{eqnarray*}
|(D_2 p-p D_2)(\gamma)| & \leq & \sum_{k=1}^\infty |\gamma_{k+1} - \gamma_{k}|\chi_{[2k,2k+1)}(t) \\
& \leq & \sum_{k=1}^\infty M_k \chi_{[2k,2k+1)}(t)
\end{eqnarray*}
where $M_k := \sup_{t \in [k,k+1)}g(t) - \inf_{t \in [k,k+1)}g(t)$.
Hence $\phi \circ L^{-1}(|(D_2 p-p D_2)(\gamma)|)
\leq \limsup_k M_k = 0$ by Lemma \ref{lemma:tech0_6}.
Thus
$$
\mathcal{L}(\bl)(D_2(\gamma)) = \phi \circ L^{-1} \circ D_2 \circ p(\gamma)
= \phi \circ L^{-1} \circ p (\gamma) = \mathcal{L}(\bl)(\gamma)
$$
by dilation invariance of $\phi \circ L^{-1}$.
\end{proof}


\subsubsection*{Proof of Theorem}

For $\bl \in BL \cap DL$ set $\phi := \bl \circ E$.  Then $\phi$ satisfies the properties (i),(ii),(iii),(iv),(iv').  By Remark \ref{rem:BLsuff}
we have the conclusion of (\cite{CRSS}, Thm 4.11),
\begin{equation} \label{eq:thm4.11}
\phi \circ L^{-1} \left( \frac{1}{\ln(1+t)} \int_1^t \mu_{\floor{s}}(PTP)ds \right)
=  \phi \left( \frac 1t \Tr((PTP)^{1+\frac 1t}) \right) .
\end{equation}
By Lemma \ref{prop:projcont} the rhs is equal to
$\phi ( 1/t \Tr(PT^{1+1/t}P))$.
By Lemma \ref{lemma:tech0_4} we have $\phi \left(1/t \Tr(PT^{1+1/t}P) \right) = \bl \left(1/k \Tr(PT^{1+1/k}P) \right)$ (*).
As before, define $g(t) := \ln(1 + t)^{-1} \int_1^t \mu_{\floor{s}}(PTP) ds$, $t \geq 1$ ($0$ otherwise). By Lemma \ref{lemma:tech0_6}, which is similar to (\cite{LSS}, Prop 2.12),
$\phi \circ L^{-1} \left( g(t)
- pr(g)(t) \right) = 0$.
Hence
$\phi \circ L^{-1} (g(t))
= \mathcal{L}(\bl)(\{g(n)\}_{n=1}^\infty)= \Tr_{\mathcal{L}(\bl)}(PTP)$ (**).
From (\ref{eq:thm4.11}), (*), and (**), we have shown
$\Tr_{\mathcal{L}(\bl)}(PTP) = \bl \left(1/k \Tr(PT^{1+1/k}P) \right)$.

Set $h(t) = 1/t\Tr(PT^{1+1/t}P)$.
Suppose $PTP$ is measurable, then $\lim_{t \to \infty} h(t)$
exists by Theorem \ref{thm:resPcont}. Hence
$\lim_{k \to \infty} h(k)$ exists.  Note
$\bl\left(1/k \Tr(PT^{1+1/k}P) \right)$ equals this limit
as $\bl$ is a generalised limit.

Conversely, suppose $\lim_{k \to \infty} h(k)$ exists.
Note
$\lim_{n \to \infty} \sup_{t \in [n,n+1)} |h(t)-h(n)| \leq
\lim_{n \to \infty} (\sup_{t \in [n,n+1)} h(t) - 
\inf_{t \in [n,n+1)} h(t)) = 0$ by the proof of Lemma \ref{lemma:tech0_2}.  Hence $\lim_{t \to \infty} h(t)$
exists and the limits are equal.
By Theorem \ref{thm:resPcont} $PTP$ is measurable. \qed

\subsubsection{Proof of Corollary \ref{cor:resA}}

The proof is identical to that of Corollary \ref{cor:resAcont}.
The equality of $\lim_{k \to \infty} k^{-1}\zeta_{A,T}(1+k^{-1})$
with $\lim_{s \to 1^+} (s-1)\zeta_{A,T}(s)$ is contained in the last paragraphs.

\providecommand{\bysame}{\leavevmode\hbox to3em{\hrulefill}\thinspace}
\providecommand{\MR}{\relax\ifhmode\unskip\space\fi MR }
\providecommand{\MRhref}[2]{%
  \href{http://www.ams.org/mathscinet-getitem?mr=#1}{#2}
}
\providecommand{\href}[2]{#2}

\end{document}